\documentclass[12pt]{article}
\RequirePackage{amsfonts,amssymb,amsbsy,amsmath}
\usepackage{epsfig,longtable}
\usepackage{url}
\def\no{\noindent}

\def\dd{\,\mathrm{d}}
\def\QED{~\mbox{$\Box$}}

\newtheorem{definition}{Definition}
\newtheorem{theorem}{Theorem}

\newtheorem{lemma}{Lemma}
\newtheorem{corol}{Corollary}
\newtheorem{remark}{Remark}

\def \RR {{\mathbb{R}}}

\def \no {\noindent}

\def \pmatrix{ \left( \begin{array} }
\def \endpmatrix{ \end{array} \right) }
\begin{document}

\date{\small August 27, 2010.}
\title{Numerical Solution of ODEs and the  {\em  Columbus'\,Egg}:\\ Three Simple Ideas for Three Difficult Problems.
\thanks{Work developed within the
project ``Numerical methods and software for differential
equations''.}}

\author{Luigi Brugnano\thanks{Dipartimento di Matematica
``U.\,Dini'', Universit\`a di Firenze, Italy ({\tt
luigi.brugnano@unifi.it}).} \and Felice
Iavernaro\thanks{Dipartimento di Matematica, Universit\`a di Bari,
Italy ({\tt felix@dm.uniba.it}).} \and Donato
Trigiante\thanks{Dipartimento di Energetica ``S.\,Stecco'', Universit\`a di
Firenze, Italy ({\tt trigiant@unifi.it}).}}

\maketitle

\vspace{-1cm}

\begin{abstract}
On computers, discrete problems are solved instead of continuous
ones. One must be sure that the solutions of the former problems,
obtained in real time (i.e., when the stepsize $h$ is not
infinitesimal) are good approximations of the solutions of the
latter ones. However, since the discrete world is much richer than
the continuous one (the latter being a limit case of the former),
the classical definitions and techniques, devised to analyze the
behaviors of continuous problems, are often insufficient to handle
the discrete case, and new specific tools are needed. Often, the
insistence in following a path already traced in the continuous
setting, has caused waste of time and efforts, whereas new
specific tools have solved the problems both more easily and
elegantly.

In this paper we survey three of the main difficulties encountered
in the numerical solutions of ODEs, along with the novel solutions proposed.

\medskip \no {\bf AMS:} 65P10, 65L05.

\medskip \no{\bf Keywords:} ordinary differential equations, linear
multistep methods, boundary value methods, $A$-stability,
stiffness, Hamiltonian problems, Hamiltonian Boundary Value
Methods, energy preserving methods, symplectic methods, energy
drift.

\end{abstract}

\section{Introduction}
\begin{flushright}
  \textsf{\em When things get too complicated, it sometimes makes sense\\
              to stop and wonder: Have I asked the right question?\quad~~~~~~~}\\
  \textsf{Enrico Bombieri, from `Prime Territory' in The Sciences.~~~~~~}\\
\end{flushright}

The need for highly efficient numerical methods able to solve the
challenging multiscale problems arising from countless and
wide-spread applications is well known (see, e.g. \cite{multrep}).
The core of a simulating tool for differential models consists of
one or more numerical methods for solving Ordinary Differential
Equations (ODEs).

In this paper we survey three of the main difficulties encountered
in this field along with the surprising solutions proposed, often
different from what both experience and  tradition were suggesting.
As a matter of facts, both tradition and experience in Mathematics
are mainly focused to continuous quantities, while Numerical
Analysis is obliged to face discrete quantities.

The modern numerical treatment of ODEs starts with the
introduction of  computers. The approaches of the pre-computer and
the post-computer era are quite different. While in the first case
the central concept was the notion of convergence with respect to
the annihilation of  a parameter representing the stepsize of
integration, in the second case the central concept has become the
notion of stability, although often in the  most elementary form.
The reason of this change is due to the fact that there are
convergent methods which give bad results even for very small
values of the stepsize. This was not clear when the computations
had to be made by hand and, therefore, only in limited quantity.
Once computers allowed the execution of a huge number of
operations within a small amount of time, the problem arose in all
its gravity.

This question was largely debated in the fifties and sixties. It
is a great merit of G.\,Dahlquist to recognize that the difference
equations describing the methods should inherit from the
continuous problem not only the critical points but also the
geometry around them. Although many authors date the birth of the
so called \emph{geometric integration} to a later period, it is
our opinion that it has to coincide with the mentioned Dahlquist
request.

 Actually, at
that time, the main concern was about dissipative problems whose
critical points are obviously asymptotically stable. Since for
such problems the linearization in a neighborhood of the critical
point very well describes the geometry around it, the use of the
famous test equation ($y'=\lambda y$, $\Re(\lambda)<0$) is then
justified and this gave rise to the so called \emph{linear
stability analysis}.

The main result of such analysis is the definition of the region
of absolute stability, which  is the region of the $q$-plane
($q=h\lambda$) in correspondence of which the critical point (the
null solution) is asymptotically stable for the difference
equation obtained by applying the numerical method to the above
mentioned test problem. Of course, one may wish to obtain absolute
stability regions no smaller than the stability region of the
continuous problem ($\Re( \lambda) <0 $) and this led to the
definition of $A$-stable methods.

This request very soon turned out to be too restrictive, at least
for the class of Linear Multistep Methods (LMMs),  which the
negative result of the Dahlquist barrier refers to: there are no
$A$-stable explicit methods, and among the implicit ones, the best
is the trapezoidal rule, which is only second order.

This shortage of $A$-stable methods explains the period of crisis
in the use of LMMs and the upper hand of one-step methods
(Runge-Kutta) over them. Eventually, in the nineties, the way to
obtain $A$-stable LMMs was found and by now a very large amount of
them are at our disposal. The idea to get them,  already  proposed
by Dahlquist \cite[p.\,378]{DAbook}, is very simple and it will be
our first topic.

Do all ODEs need large absolute stability regions? of course not.
But difficult problems,  called \emph{stiff problems}, do. What are
stiff problems? A  definition, precise enough to be used
  in  modern general-purpose codes, has been lacking for
 many years.\footnote{Even now, one can read, for example on Scholarpedia \cite{Higu},
that such a definition is not possible.} As matter of fact, a proper
definition of stiffness exists since 1996 and it has also been
usefully used in many modern codes. This relies upon a simple idea
which will be our second topic.

What about conservative  problems? Here the geometrical properties
the numerical method has to inherit are more difficult to
establish, since they are much more perturbation dependent than
dissipative systems. In fact, after Poincar\'e (see, for example,
\cite{Poincare-1886,AIR}), it is well known that linearization
does not help in this case, the literature being plenty of
examples of ODEs sharing the same linear part, with the geometry
around the stable (but not asymptotically stable) critical points
changing drastically according to the nonlinear part. This implies
that a linear stability analysis on the method does not make
sense, unless one is interested in solving linear problems.

There are, however, other peculiarities which could be desirably
inherited by the numerical method. For example, in a Hamiltonian
problem describing the motion of an isolated mechanical system, one
may ask to preserve a number of constants of motion, such as the
Hamiltonian function, the angular momentum, etc.. Although the
question has been under study for more then thirty years, no really
useful steps forward have been made until recently. Usually, people
have tried to mimic what physicists have done in the continuous case
without impressive results, in spite of the great amount of work.
For example, very sophisticated tools like backward analysis and
KAM-theory have been considered to examine the long time behavior of
the numerical solution produced by symplectic and symmetric methods.

Nonetheless, there are a few exceptions where completely new
strategies have been considered. One of these is represented by
{\em discrete gradient methods}, which are based upon the
definition of a discrete counterpart of the gradient operator
\cite{G,MQR}. Such methods guarantee the energy conservation of
the numerical solution whatever the choice of the stepsize of
integration. At present, methods in this class are known of order
at most two. More recently, a completely new approach has been
developed that has allowed the definition of a very wide class of
methods of any high order, suitable for the integration of
Hamiltonian problems \cite{BIT0,BIT1,BIT2,BIT3,BIT4}: this will be
our third topic.

This paper  contains three main sections, one for each theme
described above. Each section will contain a  number of subsections,
describing the problem along with the principal attempts devised to
solve it, and the idea that has inspired the new approach which,
often, turns out to be completely different from the previous ones.

We shall intentionally devote a larger part of the paper to discuss the last topic (i.e.,
the numerical integration of Hamiltonian problems), because of two reasons:
\begin{itemize}
 \item the presented results are relatively new;

\item Hamiltonian problems assume a paramount importance in modeling problems of Celestial Mechanics.
\end{itemize}

\section{$A$-stable Linear Multistep Methods} In order to set the problem in its
historical background, let us report what Hindmarsh, one of the
leading experts, wrote in a famous paper \cite{Hind}:

\begin{quote} \it
As recently as 1960, the commonly held perception on ordinary
differential equations in practical applications was that almost all
of them could be solved with simple numerical methods widely
available in textbooks. Many still hold that perception, but it has
become more and more widely realized, by people in a variety of
disciplines, that this is far to be true.
\end{quote}

Multiscale problems (in this setting called \emph{stiff problems},
see the next section) arising in a wide spectrum of applications,
very soon made the \emph{ known methods}, based on the concept of
convergence for {\em $h$ approaching  zero}, inadequate. The new
idea was to design methods working for {\em finite values of $h$},
at least for dissipative problems. This led to the already
mentioned definition of $A$-stability. But the Dahlquist barrier
established that the class of LMMs  having such property is very
scanty, or even empty if referred to the class of  methods of
order greater than $2$. This situation lasted for at least thirty
years. But it was again Dahlquist who had foreseen the possible
solution, as clearly stated in \cite[p.\,378]{DAbook}, regarding
$k$-step LMMs:

\begin{quote} \it
\dots when $k>1,$ the $k-1$ extra conditions to define a solution of
the difference equation, need not to be initial values. One can also
formulate a boundary-value problem for the differential equation.
The boundary-value problem can be stable for a difference equation
even though the root condition is not satisfied. This has been
pointed out in an important article by J.C.P. Miller.
\end{quote}

\subsection{The new idea: Miller's algorithm}

Why did the above precognition remain sterile for many years? The
successes in this direction obtained within the flourishing class of
Runge-Kutta methods somehow determined an unfavorable climate for
Miller's idea to inspire a systematic analysis. But, in our opinion,
this is not the whole story. It was also due to the concept of
stability   in Numerical Analysis, which was very vague at that
time. The most rigorous definition of such concept was related to
linear difference equations coupled with all initial conditions
(initial value problems).\footnote{\label{rootcon}With very few
exceptions, in Numerical Analysis books, stability is equivalent to
the so called \textit{root condition}, which requires that all the
roots of the characteristic polynomial of the difference equation
obtained by applying the LMM to the test equation $y'=0$ (or to
$y'=\lambda y$) lie inside the unit circle of the complex plane or
on its boundary, in which case they must be simple.} It was
necessary to refine the concept of stability in the quoted Dahlquist
phrase. We shall see in a moment that the \emph{stability} in
Miller's algorithm is in contradiction with the \emph{stability}
used in most numerical analysis books and reported in the footnote.
In order to be clearer, let us describe Miller's algorithm in its
easiest form.

Consider the linear difference equation with initial values
\begin{equation}\label{Diffeq}
  y_{n+2}=100.5y_{n+1}-50y_n, \qquad n\ge0,  \qquad y_0=\sqrt{3};\quad
  y_1=.5\sqrt{3}.
\end{equation}
Its solution is $y_n=2^{-n}\sqrt{3}$. However, when using finite
precision arithmetic, every computer, no matter how powerful, will
fail to compute iteratively the solution of
(\ref{Diffeq}),\footnote{If, accidentally, a computer does the job,
just replace the couple of coefficients $(100.5,50)$ by
$(1000.5,500)$. The solution remains the same.} despite its
simplicity. The algorithm designed by Miller, is able to find  a
very good approximation of the  solution even on the poorest
computer. It works as follows. Suppose we are interested in the
solution for $n<10$: just replace the second condition by
$y_{10}=0$. This transforms the original initial value problem (IVP)
in a boundary value problem (BVP). In Figure \ref{figmiller}, three
solutions are represented: two of them, the BVP solution and the
true one, are almost indistinguishable; the third one, i.e., that
obtained iteratively, accumulates errors and very soon becomes
negative.\footnote{If instead of $n=10$ one takes $n=20$, the first
two solutions remain positive and indistinguishable, while the third
one reaches the impressive value $-10^{20}$.}

\begin{figure}
\centerline{\includegraphics[width=12cm,height=9cm]{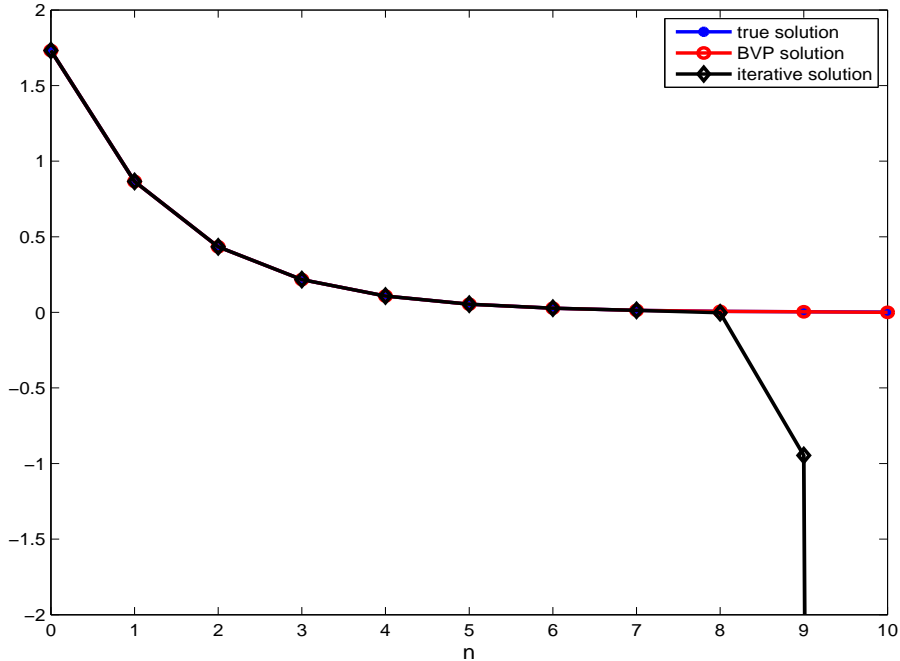}}
\caption{Solution of problem (\ref{Diffeq}) computed by Miller's
algorithm.} \label{figmiller}
\end{figure}

Even using the term ``stability''  in the vague definition often
used in Numerical Analysis, it is evident that the BVP solution is
much more stable than the IVP solution. To be more precise, the
characteristic polynomial of (\ref{Diffeq}) has one root inside
the unit circle and one outside. It cannot be considered stable
 according to the definition reported in Footnote~\ref{rootcon}. But
why does it turn out to be stable for the BVP? The answer is simple:
the definition of stability only concerns  IVPs. For BVPs one must
use the more general concept of  \emph{conditioning}. We do not
report here the details (see, e.g. \cite{imt}).  Once this new and
more precise concept is used, then all the above questions become
 clear. Coming back to the above example, it turns out that a
difference equation with one root inside and the other outside the
unit circle is well conditioned if the two conditions that guarantee
the uniqueness of its solution are placed one at the initial point
and the other at the final point. This is enough not only to explain
the result of the example, but also to explain, for example, the
poor performances of the celebrated \emph{shooting method} for
solving continuous BVPs (see \cite{imt}).

Miller's algorithm was deeply studied later on (see, e.g.
\cite{WG,Cash1,Olv1,Olv2,Ma1,Ma2}) and eventually applied to design
methods for ODEs, according to Dahlquist's suggestion. For details,
see \cite{BTbook} and  references therein.

\subsection{Abundance of $A$-stable methods}

 Once the vague concept of stability in Dahlquist's proposal
  became more precise with
 the introduction of the \emph{conditioning} notion,
 the definition of a great variety of  $A$-stable LMMs followed.
Of course, the  definition  of  absolute stability, which requires
that the asymptotic behaviour of the numerical solution be the same
as that of the theoretical one, in relation to the test equation
$y'=\lambda y$, $\Re(\lambda)<0$, needed to be generalized
accordingly.

\begin{definition}
 A $k$-step LMM coupled  with $k_1$ initial conditions and $k_2$ final
 conditions ($k=k_1+k_2$),  is  absolutely stable at $q=h \lambda$, if  the characteristic
 polynomial has $k_1$ roots inside the unit circle and $k_2$ roots
  outside.\footnote{Note that the definition reduces to the old one when $k_2=0,$ i.e. for LMMs coupled with only initial conditions.}
\end{definition}

  In other words,
 the requirement that all the roots of the characteristic polynomial
 of the difference equation defining the methods must be inside
 the unit circle was too much restrictive. The freedom to have
 some of them outside the unit circle opens wide possibilities. The
 only cost to pay is to shift some
 additional conditions to  the end of the interval. Such methods have been called
 Boundary Value Methods (BVMs) (see, e.g. \cite{BTbook}).
We are now plenty of $A$-stable linear multistep methods: even
the methods with the highest possible order with respect to the
given number of steps (Top Order Methods) are  $A$-stable.

The above result has made it possible to write efficient codes for
stiff (multiscale) problems based on LMMs (the codes GAM, GAMD,
GAMP \cite{im1, im2} BIM, BIMD \cite{BM04,BM07}, and TOM. See also
the {\em Test Set for IVP Solvers} \cite{testset} for benchmarks
of the most popular codes for ODEs).

\section{Stiffness}
Stiffness is a mathematical term to denote multiscale problems. At
least, this was the first meaning of the term, although during the
years the terms has been used in a great variety of meanings (for a
historical account see \cite{BMT}). A common way this term is
perceived currently, is the following:  a stiff problem is a problem
for which \emph{explicit methods do not work}. It is glaringly
evident the inadequacy of such definition both for the mathematical
needs and for logical reasons.

Despite the mentioned  variety of definitions, the design of
powerful codes needs a very precise definition of  stiff problems,
in order to be able to automatically recognize them  and choose
the appropriate strategies (appropriate methods, mesh selections,
etc.). The situation assumes a paradoxical aspect: from the one
hand some experts   claim the impossibility of giving the needed
precise definition of stiffness, and from the other, such
definition does exist and is also published in a book
\cite[pp.\,237 ff.]{BTbook}. The reason of such reluctance in
accepting the mentioned precise definition stays, in our
optimistic opinion, in the fact that the definition is based on a
very simple idea, as we shall see soon. Fortunately, things are
changing in recent years since the new definition has  been  used
to improve the performance of some celebrated codes
\cite{CaCaMa07,CaMa05,CaMa06,CaMa09}.

\subsection{The simple new idea}

The two plots in Figure~\ref{stiffig} report two functions: one
constant and the other rapidly variable. How is it possible to
distinguish their behaviors without looking  at them? Simply, just
compute the areas under their graphs and compare the results. Let
$T$ be the interval, for technical reasons we normalize the two
areas by dividing them by $T$. We get, respectively,
\begin{equation}\label{kgam}
\kappa_c=\frac{T|y_0|}{T}=|y_0|; \quad
\gamma_c=\frac{1}{T}\int_0^T|y(t)|dt.
\end{equation}
\begin{figure}
\centerline{\includegraphics[width=7.5cm,height=7.5cm]{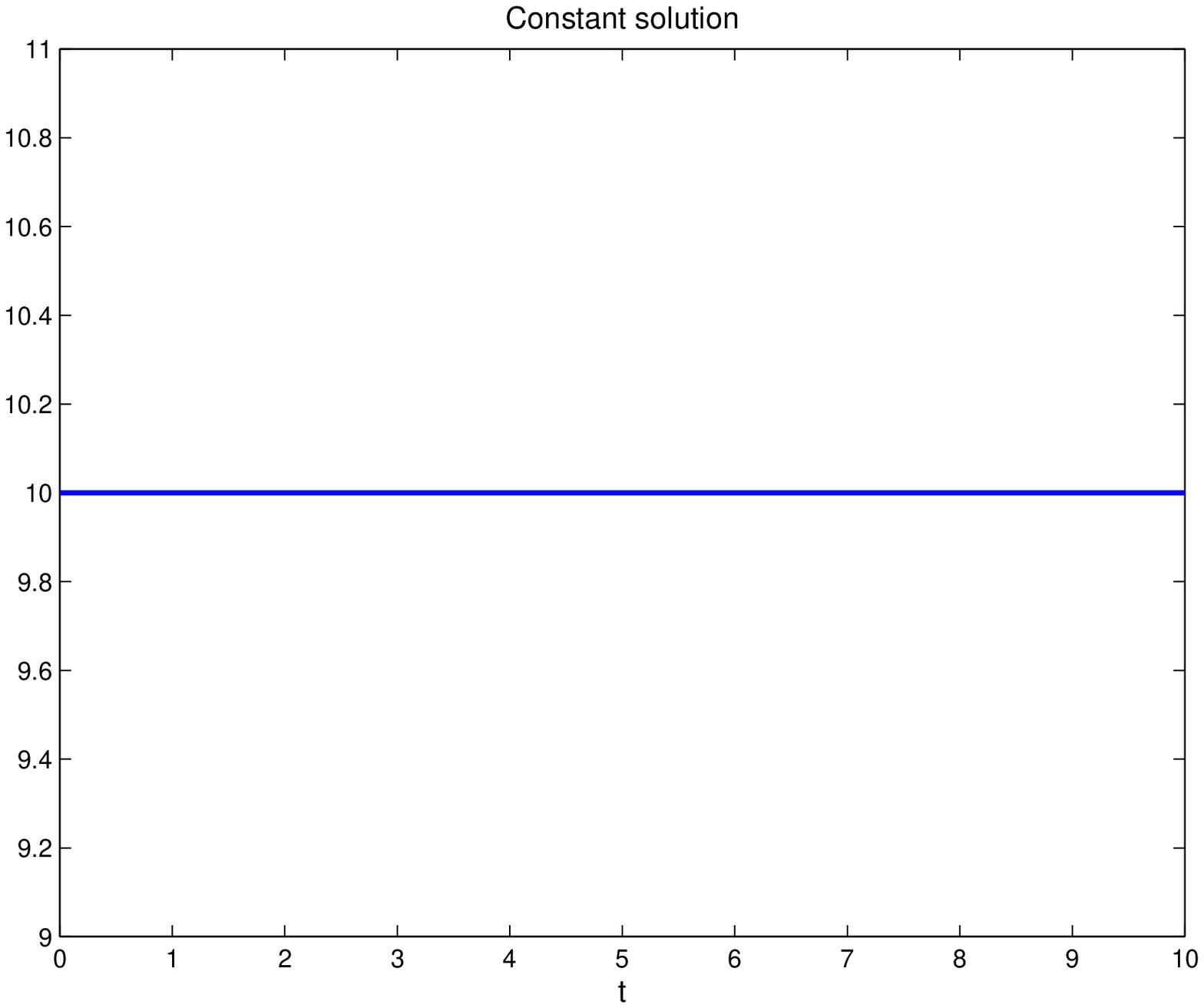}\qquad
\includegraphics[width=7.5cm,height=7.5cm]{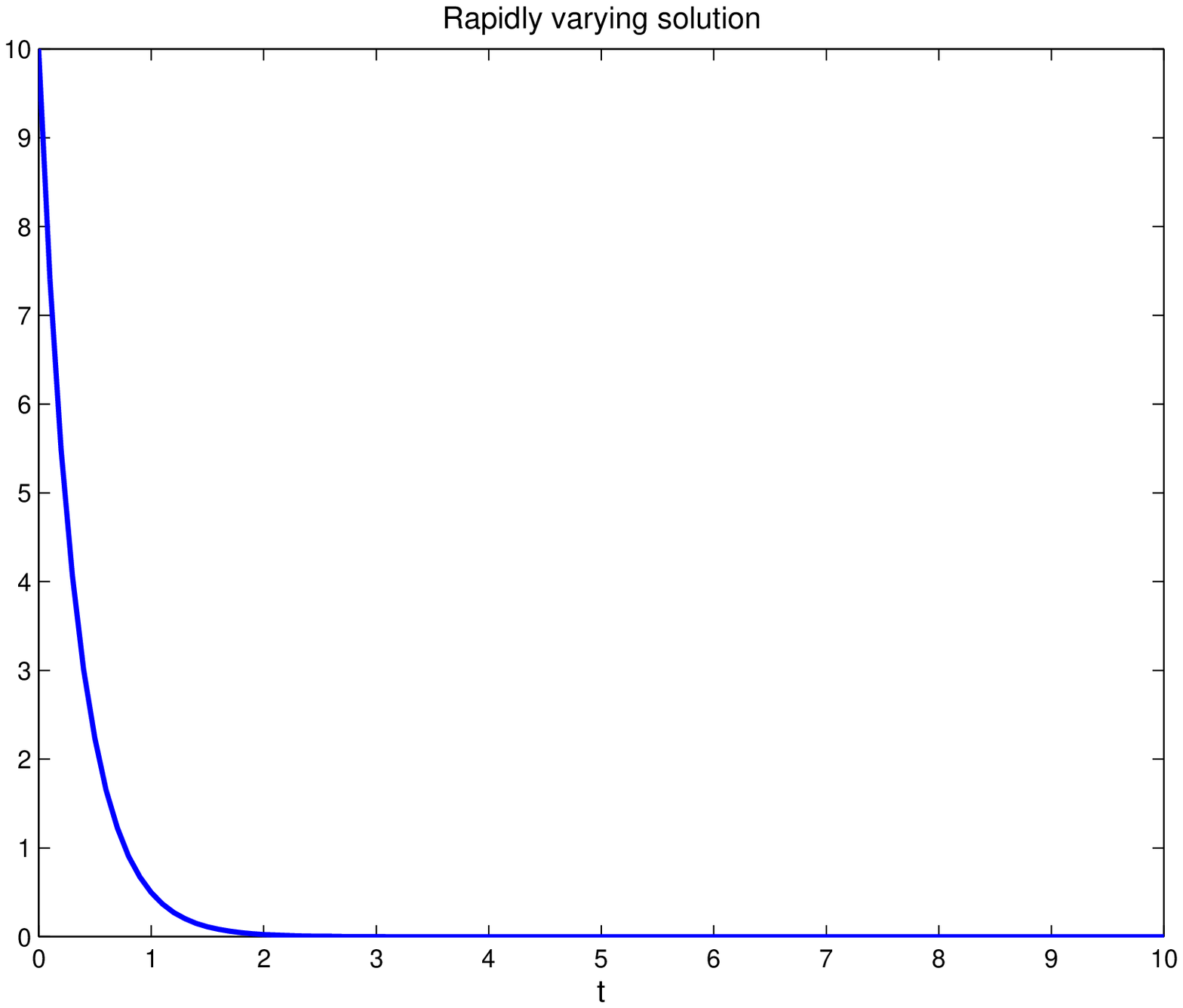}}
\caption{Functions with different behaviors: constant (left plot) and rapidly varying (right plot).}
\label{stiffig}
\end{figure}

Suppose that $y(t)=e^{\lambda t}y_0$ with $\lambda$ negative. We get:
$$
\gamma_c=\frac{1-e^{\lambda T}}{T |\lambda|}|y_0|<
\frac{1}{T|\lambda|}|y_0|.
$$

The ratio $$\sigma_c=\frac{\kappa_c}{\gamma_c} \simeq
\frac{T}{\frac{1}{|\lambda|}}$$ is the ratio of two times: the
integration interval ($T$) and the decaying time
($\frac{1}{|\lambda|}$). Of course, the important case is the one
in which the two types of solutions occur in the same problem. The
quantity $\sigma_c$ measures the stiffness (or the multiscalarity)
of the problem. The simple definition given above has been
generalized to more general  differential problems, autonomous and
non autonomous dissipative problems, conservative problems, etc.
(see \cite{BMT,BTstiff,BT07,BTbook,imt}). Of course the extension
of the definition to more general problems needs the introduction
of more involved technicalities, but the leading idea remains
unchanged: the definition of stiffness needs two measures, such as
the infinity norm ($\kappa_c$) and the $L_1$ norm ($\gamma_c$),
and the ratio between the two.\footnote{Actually, the maximum over
an appropriate set of ratios.}

The above definition of stiffness deals with the continuous case. Of
course, a very similar definition can be introduced for discrete
problems, thus leading to  the corresponding parameters $\kappa_d$,
$\gamma_d$, $\sigma_d$. The two set of parameters (called
\emph{conditioning parameters}), turn out to be useful also to tell
if the discrete problem is appropriate for approximating the
corresponding continuous one. In fact one has:
\begin{definition}
A continuous problem is {\em well represented} by a discrete one if
$\kappa_c\simeq \kappa_d$ and $\gamma_c\simeq \gamma_d$.
\end{definition}
The use of the above precise definition, based on computable
conditioning parameters, has allowed:
\begin{itemize}
  \item[-] to automatically recognize stiff problems;
  \item[-] to define efficient variable mesh selection strategies.
\end{itemize}
As was pointed out above, some existing codes and new ones take
advantage from the use of these conditioning parameters
\cite{BT07,CaCaMa07,CaMa05,CaMa06,CaMa09,CaMa,Ma03,MaSeTr09,MaTr04}.

\section{Hamiltonian Problems} Hamiltonian problems form a subclass
of conservative problems. They assume the form
\begin{equation}\label{ham}
  \frac{dy}{dt}= J \nabla H(y), \quad
  J=\pmatrix{cc}0&I_m\\-I_m&0\endpmatrix, \quad y(t_0)=y_0,
\end{equation}

\noindent where $y=(q^T,p^T)^T$, $q,p\in \RR^{m}$, and $H$ is a
sufficiently smooth scalar function.

As was observed in the introduction, the main difficulty in dealing
with them numerically  stems from the fact that the meaningful
isolated critical points of such systems are only marginally stable:
neighboring solution curves do not eventually approach the
equilibrium point either in future or in past times.

This implies that the geometry around them critically depends on
perturbations of the linear part. Consequently, the use of a
linear test equation, which essentially captures the geometry of
the linear part, whose utility has been enormous in settling the
dissipative case, cannot be of any utility in the present case.

It is then natural to look for other properties of Hamiltonian
systems that can be imposed on the discrete methods in order to
make them efficient. The first property which comes to mind is the
symplecticity  of the flow $\varphi_t := y_0 \mapsto y(t)$
associated with \eqref{ham}. This property can be described either
in geometric form (invariance of areas, volumes, etc.) or in
analytical form: $$\left(\frac{\partial \varphi_t}{\partial
y_0}\right)^T J \left(\frac{\partial \varphi_t}{\partial
y_0}\right)=J.$$ In one way or the other, it essentially consists
in moving infinitesimally on the trajectories representing the
solutions. Infinitesimally means retaining only the linear part of
the infinitesimal time displacement $\delta t$. It can be shown
that this produces  new values of the variables $q+\delta q,$
$p+\delta p$ which leave unchanged the value of the Hamiltonian
$H(q+\delta q,p+\delta p)=H(q,p)$ (Infinitesimal Contact
Transformation (ICT, see \cite[p.\,386]{Gold})).

Consequently, since the composition of two or more  such
\textbf{infinitesimal} transformations maintain the invariance, so
does an \textbf{infinite number} of them. The beauty of such result
in the continuous case has perhaps created the expectation that
similar beautiful achievements could be obtained in the discrete
case. It is not surprising that the first attempts, to numerically
approach  the problem, aimed at devising symplectic integrators
(\cite[Ruth (1983)]{Ruth} , \cite[Feng Kang (1985)]{Feng}; see also
the monographs \cite{SC,LeRe,HLW}  for more details on the subject).

Although in particular circumstances, for example in the case of
quadratic Hamiltonians or, more in general, in the case of
Hamiltonian systems admitting quadratic first integrals,  this
approach has provided very good results,\footnote{For example, a
symplectic Runge-Kutta method precisely conserves all quadratic
invariants of the motion.} it cannot be considered conclusive in
treating general Hamiltonian problems.

A backward error analysis has shown that symplecticity somehow
improves
 the long-time behavior properties of the numerical
solutions. For a symplectic method of order $p$, implemented with
constant stepsize $h$, the following estimation reveals how the
numerical solution $y_n$ may depart from the manifold $H(y)=H(y_0)$
of the phase space:
\begin{equation}
\label{backward} H(y_n)-H(y_0) = O(nh\,e^{-\frac{h_0}{2h}} h^p),
\end{equation}
where {\em $h_0>0$ is sufficiently small} and $h\le h_0$. Relation
\eqref{backward} implies that a linear drift of the energy with
respect to time $t=nh$ may arise. However, due to the presence of
the exponential, such a drift will not appear as far as $nh\le
e^{\frac{h_0}{2h}}$: this circumstance is often referred to by
stating that  {\em symplectic methods conserve the energy function
on exponentially long time intervals} (see, for example,
\cite[Theorem 8.1]{HLW}). This is clearly a surrogate of the
definition of stability in that the ``good behaviour'' of the
numerical solution is not extended on infinite time intervals. Even
more alarming is the fact that if one wants to compute the numerical
solution over very long times (as is done, for example, in the study
of the stability of the Solar System), on the basis of
\eqref{backward}, he may be obliged to reduce the stepsize below a
safety threshold, which is in contrast with the spirit of the long
time simulation of dynamical systems where the use of very large
stepsizes is one of the primary prerogatives.\footnote{Another
constraint is that $h_0$ is assumed {\em small enough}. In our
opinion, when possible,  expressions like ``for $h$ small enough''
should be avoided in Numerical Analysis: we like to believe that
geometric integration has been devised just to eliminate such an
expression.}

Where is the weakness of the approach? It is just in the above
outlined words \textbf{infinitesimal} and \textbf{infinite}, which
should be prohibited in Numerical Analysis. This discipline, in
fact, has to deal with nonzero   (greater than machine precision)
and finite (bounded either by the patience of the operator or by
the cost of energy) quantities. In other words, following this
approach, the situation of the pre-Dahlquist era for dissipative
problems has been recreated for conservative problems, in the
sense that before Dahlquist there already existed methods, even
with high order of convergence, that for {\em $h$ small enough}
would do the job (for example, the midpoint and Simpson's
methods).

Coming back to linear problems,  not all symplectic methods provide a
conservation of the Hamiltonian function even in this simpler case:
the following  example has been taken from \cite[Example\,8.2.1 on p.\,189]{BTbook}.
Consider the harmonic oscillator problem in Hamiltonian form
\begin{equation}\label{osc}
  \frac{d}{dt}\pmatrix{c}q\\p\endpmatrix=J\pmatrix{c}q\\p\endpmatrix.
\end{equation}
Let $h$ be the integration step and consider the numerical method
defined by
\begin{equation}\label{symeth}
\pmatrix{c}q_{n+1}\\p_{n+1}\endpmatrix=M_h\pmatrix{c}q_{n}\\ p_{n}\endpmatrix
\quad M_h=I+hJ-h^2\pmatrix{cc}1 &0\\0&0\endpmatrix.
\end{equation}
Since the continuous solution in $h$ is $e^{hJ}y_0,$ it is not
difficult to deduce that the method is first order:
$e^{Jh}-M_h=O(h^2)$. The method is symplectic, since
$M_h^TJM_h=J$, but fails to be conservative since, considering
that $H(q,p)=(q,p)(q,p)^T/2$, we have
$$
(q_{n+1},p_{n+1})\pmatrix{c}q_{n+1}\\p_{n+1}\endpmatrix=
(q_n,p_n)M_h^TM_h\pmatrix{c}q_{n}\\p_{n}\endpmatrix\neq
(q_n,p_n)\pmatrix{c}q_{n}\\p_{n}\endpmatrix.
$$

The matrix $M_h$ is orthogonal only if the term  $h^2$ is not
present, according to the ICT hypothesis. But, unfortunately, this
cannot be accepted since we want to use the method with finite
values of $h$.

The literature about symplectic integrators is quite wide since
Numerical Analysts, Physicists and Engineers have been working on
them since more than 25 years. Consequently, this testifies their
importance in the applications. However, other approaches have
been attempted, among which:

\begin{itemize}
\item[-] projection methods \cite{AR,H} and numerical integrators on
manifolds (see, e.g., \cite[Sect. IV.5.3]{HLW});
\item[-] definitions of discrete counterparts of operators describing
conservative vector fields,
  such as discrete gradients \cite{G,MQR}, discrete divergence \cite{IT1},
  averaged vector fields
  \cite{QMcL,Ha}, discrete line integrals \cite{BIT0} (see the next section
  for an introduction to the latter methods).
\item[-] generalizing the definition of symplecticity so as to
include some nonlinearity in it (state dependent symplecticity
\cite{IT2}).
\end{itemize}

\begin{remark}
It is worth mentioning that a Hamiltonian system may have other
constants of motion, consequently the following question arises:
suppose that we are able to devise methods conserving the energy,
are the methods also able to preserve, for example,  quadratic
invariants? Few results have been presented so far regarding
essentially methods in the Runge-Kutta class, most of them rather
pessimistic. The paper \cite{BIT4} gives a first positive answer to
this issue.
\end{remark}

 From the above discussion it turns out that the problem
considered  is a very difficult one, although very important in the
applications. The only solid result obtained in all these years
seems  to be the one establishing that, in order that a method can
preserve the Hamiltonian functions, it must be symmetric (see, e.g.,
\cite{BTbook}), although, of course, this condition in not
sufficient.

\subsection{The simple new idea} \label{easy} The novel approach
(see, for example \cite{BIT0} and references therein) starts from
a trivial observation. Let $\omega(t)$ be any smooth curve
$[t_0,t_1] \rightarrow \RR^{2m}$ passing through $y_0$ at time
$t_0$. Then, we have:

\begin{equation}\label{obs}
  H(\omega(t_1))-H(y_0)=\int_{t_0}^{t_1}\frac{d}{dt}H(\omega(t))dt.
\end{equation}

Of course, choosing $\omega(t)$ as  the solution $y(t)$ to
\eqref{ham} yields
$$
H(y(t_1))-H(y_0)=
\int_{t_0}^{t_1}\nabla^T H(y(t)) \frac{d y(t)}{dt}
\mathrm{d}t=\int_{t_0}^{t_1}\nabla^TH(y(t))J\nabla
 H(y(t))\mathrm{d}t=0.
$$
The above result, which implies the conservation of the Hamiltonian
function along the trajectory $y(t)$ at times $t_0$ and $t_1$, has
been obtained by exploiting \eqref{ham} and the skew-symmetry of the
matrix $J$.

Is it possible to obtain a similar result for a curve $\omega(t)$
not coincident with the unknown solution $y(t)$ but nonetheless
approximating it to a given order?

Surprisingly enough, the answer is positive. Let
$\{\omega_j(t)\}_{j=0}^{\infty}$ be a set of linearly independent
scalar functions defined on $[t_0,t_1]$ and
$\{\gamma_j\}_{j=0}^{r-1}$ a set of unknown vectors. For
simplicity, we assume that
\begin{itemize}
  \item[(1)] the interval $[t_0,t_1]$ coincides with $[0,1]$;
  \item[(2)] the functions $\{\omega_j(t)\}_{j=0}^{\infty}$ are
  orthogonal;\footnote{Some further simplification can be obtained by
  choosing an orthonormal basis \cite{BIT0}.}
  \item[(3)] the integrals of such functions are easily expressible as
  linear combination of themselves.\footnote{This is not a severe restriction,
  since elementary
  trigonometric functions and polynomials do have such a
  property. }
\end{itemize}
By setting $\omega'(t)= \sum_{i=0}^{r-1}\gamma_i \omega_i(t)$, we
obtain
\begin{equation}
\label{omegat}
\omega(t)=y_0+\sum_{i=0}^{r-1}\gamma_i\int_0^t\omega_i(\tau)\mathrm{d}\tau\equiv y_0+\sum_{i=0}^{r'}\tilde{\gamma_i}\omega_i(t).
\end{equation}
The vectors $\tilde{\gamma_j}$, $j=0,\dots,r'$ are uniquely
determined in terms of linear combinations of the $\gamma_i$,
$i=0,\dots,r-1$,  according to the specific relations mentioned in
item (3) above.\footnote{For example, for the shifted Legendre
polynomials that we shall consider later, one has: $r'=r$.}
Setting $y_1=\omega(1)$ yields
$$
H(y_1)-H(y_0)=\sum_{i=0}^{r-1}\gamma_i^T\int_0^{ 1}\omega_i(t)\nabla
H(\omega(t))\mathrm{d}t.
$$

We now choose the $\{\gamma_i\}$ as the Fourier coefficients of
the function $J\nabla H(\omega(t))$, i.e., we pose

\begin{equation}\label{four}
\gamma_i=\eta_i\int_0^1\omega_i(t)J\nabla H(\omega(t))dt, \qquad
i=0,\dots,r-1,
\end{equation}
\noindent where $\{\eta_i\}$ are scalars that normalize the
orthogonal functions $\{\omega_i(t)\}$ in order to make them an
orthonormal basis for the space $L_2([0,1])$ of square-integrable
functions. Considering the way the vectors $\{\gamma_i\}$ are
involved in the definition of $\omega'(t)$, we see that
\begin{equation}
\label{etai} \eta_i = \left(  \int_0^1 \omega_i^2(t) \mathrm{d} t
\right)^{-1}, \qquad i\ge0.
\end{equation}

\begin{theorem}
With the choice (\ref{four}), the Hamiltonian assumes the same
values at $y_0$ and at $y_1$.
\end{theorem}

\begin{proof} We have
$$
H(y_1)-H(y_0)=\sum_{i=0}^{r-1}\eta_i\left[\int_0^1\omega_i(t)(\nabla
H(\omega(t)) )^Tdt\right]J^T\left[\int_0^1\omega_i(t)\nabla H(\omega(t))dt\right] = 0,
$$ due to the fact that $J$ is skew-symmetric.\,\QED
\end{proof}

We have then proved that the Hamiltonian can be preserved on
curves different from the solution. A rescaling of the form
$[0,1]\rightarrow [t_0,t_0+h]$ will introduce the stepsize $h$ and
the iteration of the procedure will cover all the integration
interval of interest with the result that on each interval of
length $h$ there are at least two points (the end points) where
the Hamiltonian assumes the same value.

The $r$ relations \eqref{four} define the unknown vectors $\{\gamma_i\}$
implicitly, since they appear as part of the integrands via the
curve $\omega(t)$ in \eqref{omegat}. Thus, in the general case,
\eqref{four} have to be regarded as nonlinear integral equations.
Therefore, in order to obtain a concrete numerical method we need:

\begin{itemize}
  \item[(a)] to substitute the integrals with discrete sums without
  introducing errors in the quadrature step;
  \item[(b)] to design an algorithm to solve the resulting, usually nonlinear, system
  providing the $\{\gamma_i\}$;
  \item[(c)] to check that the vector $y_1=\omega(t_1)$, which will
  be assumed as an approximation, at time $t_1$, of the true solution $y(t_1)$, has
  indeed the desired order of accuracy, say $p$: $||y(t_1)-y_1||=O(h^{p+1})$, for
  a given integer $p\ge 1$.
\end{itemize}
More in general, step (a) could be replaced by the assumption that
primitive functions for the integrands are available in closed
form, no matter whether they can be expressed via standard
quadrature formulae (see the examples  in \cite[Section 4]{BIT2}).
In any event, the requirement (a) can be completely fulfilled if
the functions $\{\omega_i(t)\}$ and $H$ are
polynomials.\footnote{Since polynomials can approximate regular
functions within any degree of accuracy, the present approach
works fine in more general contexts. Furthermore, another
important case where issue (a) is fulfilled is that of
trigonometric functions over one period (see \cite[p.
155]{Gaubook}).}

Let $\nu$ be the degree of $H(y)$, and $r-1$ be the degree of
$\omega'(t)$. Then the integrands in \eqref{four} are
polynomials as well, of degree at most $\mu=r(\nu-1)+r-1=\nu r-1$. We need then
quadrature formulae having degree of precision greater than or equal
to $\mu$. Of course, it will be advantageous to chose them of Gaussian
type.

As we will see in Subsection \ref{nonlinearcase}, the resulting
methods fall in the class of block-BVMs and have been called
Hamiltonian Boundary Value Methods (HBVMs).\footnote{If desired,
these methods also admit a Runge-Kutta formulation (see, for
example, \cite{BIT1,BIT2,BIT3}).}

To describe how the above three tasks can be accomplished, we start
with the simpler case where $H(y)$ is a quadratic function and
therefore problem \eqref{ham} is linear. In particular, we will
consider hereafter the harmonic oscillator problem realizing that
the obtained method is in fact the Gauss-Legendre method. The approach will
be then generalized to nonlinear Hamiltonians thus leading to the new
formulae.

\subsection{Application to the Harmonic Oscillator Problem}\label{lincase}

The harmonic oscillator problem is described in Eq. (\ref{osc});
note that   $\nabla H(y)=(q,p)^T$. Of course, we are tempted to take as
set $\{\omega_i(t)\}_{i=0}^\infty$ the trigonometric orthogonal
systems, since this would provide the exact solution, even for small
values of $r$.

Let us instead take for $\{\omega_i\}$ the set of shifted Legendre
polynomials $\{P_i(t)\}$, which is orthogonal in $[0,1]$. They may be
defined by the Rodrigues formula
$$
P_i(t)=\frac{1}{i!}
\frac{\mathrm{d}^i}{\mathrm{d}t^i}\left[(t^2-t)^i) \right], \qquad i\ge0
$$
(we also set $P_{-1}(t)\equiv 0$). The first few are:
$$
P_0(t)=1,\quad P_1(t)=2t-1,\quad P_2(t)=6t^2-6t+1,\quad P_3(t)=20
t^3 -30 t^2 +12 t -1.
$$
What we here need about such polynomials are the following two
properties ($\delta_{ij}$ denotes the Kronecker symbol):
\begin{itemize}
\item[$(L_1)$] $\int_0^1P_i(t)P_j(t)dt=\frac{1}{2i+1} \delta_{ij}, \qquad i,j\ge0$;
\item[$(L_2)$]
$\int_0^tP_j(x)\mathrm{d}x
=\frac{1}{2(2j+1)}(P_{j+1}(t)-P_{j-1}(t)+\delta_{j0}P_0(t)),\qquad
j\ge0$.
\end{itemize}
 Comparing Property $(L_1)$ with the normalization condition
 \eqref{etai} yields $\eta_i=2i+1$. In this example we  set  $r=3$ and,
 hence,
$$
\begin{array}{rcl}
  \omega'(t) & = & P_0(t)\gamma_0+P_1(t)\gamma_1+P_2(t)\gamma_2,
  \\[.3cm]
  \omega(t) & = &y_0+
  \frac{1}{2}(P_1(t)+P_0(t))\gamma_0+\frac{1}{6}(P_2(t)-P_0(t))\gamma_1+\frac{1}{10}(P_3(t)-P_1(t))\gamma_2.
\end{array}
$$
Relations (\ref{four}) become
$$
\begin{array}{rcl}
\gamma_1&=&\frac{30}{59}J\gamma_0, \\[.3cm]
\gamma_2&=&\frac{1}{6}J\gamma_1~=~-\frac{5}{59}\gamma_0, \\[.3cm]
\gamma_0&=&\frac{1}{6}J(3\gamma_0-\frac{30}{59}J\gamma_0+6y_0) ~=~
\frac{2\cdot 59}{4\cdot 54^2+59^2}(-59I+2\cdot 54J)y_0,
\end{array}
$$
from which, by setting $S=\frac{1}{4\cdot 54^2+59^2}(-59I+2\cdot
54J)$, we obtain
$$
\begin{array}{rcl}
  \omega'(t) & = & \left((2\cdot 59P_0(t)-10P_2(t))I
  +60P_1(t)J\right)Sy_0, \\[.3cm]
  \omega(t)
 &= &y_0+\left[(60P_1(t)+59P_0(t)-P_3(t))I+10(P_2(t)-P_0(t))J\right]Sy_0.
\end{array}
$$
and therefore the residual,
$$R(t)=\omega'(t)-J\omega(t)=P_3(t)JSy_0,$$
 is zero (collocation) when $t$ is a root of
$P_3(t)$. Since the roots of $P_3(t)$ are the abscissae of the Gauss
collocation method of order six, from the uniqueness of the
collocation polynomial we conclude that our approach applied to
linear problems leads to these  formulae.\footnote{As a matter of
fact, it is well-known that Gauss methods conserve  quadratic energy
functions.}

Interestingly, this approach leads to completely new formulae if
applied to general nonlinear Hamiltonian problems.

\subsection{Hamiltonian Boundary Value Methods (HBVMs)} \label{nonlinearcase}
Relations \eqref{four} have been retrieved by imposing the energy
conservation property at the end points of the curve $\omega(t)$,
$t \in [t_0,t_0+h]$, and we are now assuming that $H(y)$ is a
polynomial of degree $\nu$ so that the integrals may be exactly
evaluated, for example, by means of a Gaussian quadrature formula
of sufficiently high degree. As was pointed out in Subsection
\ref{easy} they form a block nonlinear system of dimension $r$
which, once solved, will provide the expression of the curve
$\omega(t)$ and, hence, of the numerical solution at time
$t_1=t_0+h$, namely $y_1=\omega(t_0+h)$. The resulting methods
have been called Hamiltonian Boundary Value Methods (HBVMs)
because they are naturally and conveniently recast as block-BVMs
(see \cite{BIT2, BIT1} for more details about their formulation
and implementation).

What about their order of convergence? Let us  consider again the
shifted Legendre polynomials. Substituting \eqref{four} into
\eqref{omegat}, and setting $c=(t-t_0)/h$, yields
\begin{equation}
\label{omegac} \omega(t_0+ch)=y_0+h \sum_{j=0}^{r-1} (2j+1)
\int_0^c P_j(x) \mathrm{d} x \, \int_0^1 P_j(\tau)
f(\omega(t_0+\tau h)) \mathrm{d} \tau
\end{equation}
and, on differentiating with respect to $c$,
\begin{equation}
\label{omegacprime} \omega'(t_0+ch)=\sum_{j=0}^{r-1} (2j+1) P_j(c)
 \, \int_0^1 P_j(\tau) f(\omega(t_0+\tau h)) \mathrm{d}
\tau.
\end{equation}
In the above relations, we have set $f(y) \equiv J\nabla
H(y)$.\footnote{In fact, the new methods also make sense for non
Hamiltonian problems.}

Both \eqref{omegac} and \eqref{omegacprime} have an interesting
interpretation. Consider the unknown solution $y(t_0+ch)$ of
\begin{equation}
\label{ode} \left\{\begin{array}{l} y'(t_0+ch)=f(y(t_0+ch)), \qquad  c\in [0,1],\\
y(t_0)=y_0,
\end{array} \right.
\end{equation}
or,  equivalently, of its integral formulation
\begin{equation}
\label{odeint} y(t_0+ch)=y_0+h\int_0^c f(y(t_0+\tau h)) \mathrm{d}
\tau.
\end{equation}
Let us consider the Fourier series of $f(y(t_0+\tau h))$ for $\tau\in[0,1]$:
\begin{equation}
\label{fourier} f(y(t_0+c h)) = \sum_{j=0}^\infty (2j+1) P_j(c)
\int_0^1 P_j(\tau) f(y(t_0+\tau h)) \mathrm{d} \tau.
\end{equation}
In terms of such expansion, \eqref{ode} and \eqref{odeint} read
\begin{equation}
\label{odet} \left\{\begin{array}{l} y'(t_0+ch)=\displaystyle
\sum_{j=0}^\infty (2j+1) P_j(c)
\int_0^1 P_j(\tau) f(y(t_0+\tau h)) \mathrm{d} \tau,\qquad  c\in [0,1],\\
y(t_0)=y_0,
\end{array} \right.
\end{equation}
and
\begin{equation} \label{odeintt} y(t_0+ch)=y_0+h
\sum_{j=0}^\infty (2j+1) \int_0^c P_j(x) \mathrm{d} x \, \int_0^1
P_j(\tau) f(y(t_0+\tau h))\mathrm{d} \tau,
\end{equation}
respectively. Consequently \eqref{omegac} and \eqref{omegacprime}
are defined by simply truncating the series on the right hand side
of \eqref{odet} and \eqref{odeintt}. Of course, in the event that
 series \eqref{fourier} actually contains a finite number of terms,
there will be no difference between \eqref{odet}-\eqref{odeintt} and
\eqref{ode}-\eqref{odeint}, provided $r$ is large enough.

How close are the two set of problems? The answer is obtained by
means of the Alekseev formula \cite{Ale}, by using the following
preliminary result.

\begin{lemma}\label{lem1}
Let $g$ be a suitably regular function and $h>0$. Then
$$\int_0^1 P_j(\tau)g(\tau h) \dd \tau = O(h^j), \qquad j\ge0.$$
\end{lemma}
\begin{proof}
Assume, for sake of simplicity, $$g(\tau h) = \sum_{n=0}^\infty
\frac{g^{(n)}(0)}{n!} (\tau h)^n$$ to be the Taylor expansion of
$g$. Then, for all $j\ge0$, $$\int_0^1 P_j(\tau) g(\tau h) \dd\tau
= \sum_{n=0}^\infty \frac{g^{(n)}(0)}{n!} h^n \int_0^1 P_j(\tau)
\tau^n\dd\tau = O(h^j),$$ since $P_j$ is orthogonal to polynomials
of degree $n<j$.\QED\end{proof}

Let us now define the functions
\begin{equation}
\label{mainpart} F_h(c,y) = \sum_{j=0}^{r-1} (2j+1) P_j(c) \int_0^1
P_j(\tau) f(y(t_0+\tau h)) \mathrm{d} \tau
\end{equation}
and
\begin{equation}
\label{residual}
R_h(c,y) = \sum_{j=r}^\infty (2j+1) P_j(c)
\int_0^1 P_j(\tau) f(y(t_0+\tau h)) \mathrm{d} \tau.
\end{equation}
 From Lemma \ref{lem1}, after setting $g(\tau h)=f(y(t_0+\tau h))$,
we deduce that
\begin{equation}
\label{residual1} R_h(c,y) = \sum_{j=r}^\infty a_j(h) P_j(c),
\end{equation}
with $a_j(h)=O(h^j)$, $j=r,r+1,\dots$ and, therefore,
$R_h(c,y)=O(h^r).$

Moreover, for any given $\tilde t\in[t_0,t_0+h]$, we denote by $\omega(s,\tilde t,\tilde y)$
the solution of \eqref{omegacprime} at time $s$ and with initial condition
$\omega(\tilde t)=\tilde y$.

\begin{lemma}[Alekseev (1961)] \label{lem2}
Consider the two initial value problems with the same initial
condition
$$\begin{array}{ll}
z'=\varphi(t,z), & z(t_0)=y_0, \\
y'=\varphi(t,y) + \psi(t,y), & y(t_0)=y_0, \\
\end{array}
$$
and suppose that $\varphi$ is continuously differentiable with
respect to the second argument. Then the two solutions $z(t)$ and
$y(t)$ satisfy the following relation:
\begin{equation}
\label{Alek} y(t)-z(t) = \int_{t_0}^t \frac{\partial z}{\partial
y_0} (t,c,y(c)) \, \psi(c,y(c)) \dd c.
\end{equation}
\begin{proof}
See, for example, \cite[Theorem\,14.5, p.\,96]{HLW} or \cite[Theorem\,7.5.1, p.\,205]{LakTri}.\QED
\end{proof}

\end{lemma}
We can now state the following result.

\begin{theorem}\label{fourierord}
Let $\omega(t_0+ch)$ and $y(t_0+ch)$, $c\in[0,1]$, be the
solutions of (\ref{omegacprime}) and (\ref{odet}), respectively.
Then,
$$y(t_0+h)-\omega(t_0+h) = O(h^{2r+1}).$$\end{theorem}
\begin{proof}
In terms of the functions $F_h$ and $R_h$, problems
(\ref{omegacprime}) and (\ref{odet}) read
$$\omega'(t_0+ch) = F_h(c,\omega), \qquad y'(t_0+ch) = F_h(c,y) +R_h(c,y),$$
respectively. From the Alekseev formula \eqref{Alek} one then
obtains, by virtue of Lemma~\ref{lem1} and \eqref{residual1}, that
\begin{eqnarray*}\lefteqn{y(t_0+h)-\omega(t_0+h) =}\\ &=&
h\int_0^1 \frac{\partial \omega}{\partial y_0}(t_0+h,t_0+\tau h,y(t_0+\tau
h))
\, R_h(\tau,y)\dd\tau \\
&=& h \sum_{j=r}^\infty  \left( \int_0^1 P_j(\tau) \frac{\partial
\omega}{\partial y_0}(t_0+h,t_0+\tau h,y(t_0+\tau h))\dd \tau \right)
a_j(h)
\\&=& h\,O(h^{r})\,O(h^{r}) = O(h^{2r+1}).\QED
\end{eqnarray*}

As a direct consequence, we obtain the following result.

\begin{corol} Let $T=Nh$ be a fixed positive real number, $N$ being an integer.
Then, the approximation to the solution of problem
$$y'(t) = f(y(t)), \qquad t\in[t_0,t_0+T], \qquad y(t_0) = y_0,$$
by means of
$$\omega'(t_{i-1}+ch) = \sum_{j=0}^{r-1} (2j+1) P_j(c) \int_0^1 P_j(\tau)
f(\omega(t_{i-1}+ch))\dd\tau, \quad c\in[0,1], \quad
i=1,\dots,N,$$ where ~$t_i = t_{i-1}+h$, ~$i=1,\dots,N$,~ and
~$\omega(t_0)=y_0$,~ is ~$O(h^{2r})$~ accurate.
\end{corol}
\end{proof}

\subsection{A numerical example}
In order to make clear the advantage of using the
energy-preserving HBVMs (\ref{omegacprime}) over standard
symplectic methods, we just mention that standard mesh selection
strategies are not advisable for symplectic methods (see, e.g.,
\cite[p.\,127]{SC}, \cite[p.\,235]{LeRe}, \cite[p.\,303]{HLW}),
since a {\em drift} in the Hamiltonian and a quadratic error
growth is experienced in such a case. The example that we consider
below gives a hint that this is not the case for HBVMs, provided
that the integral in (\ref{omegacprime}) is exactly computed (at
least, numerically which, as observed before, can be always
achieved, for all suitably regular Hamiltonian functions).

\begin{remark}\label{gaussrem} As also observed in Section~\ref{lincase},
we mention that when the integrals in (\ref{omegacprime}) are approximated
by means of the Gauss-Legendre formula with $r$ points, then the Gauss-Legendre
method of order $2r$ is obtained \cite{BIT2}.\end{remark}

\begin{figure}
\centerline{\includegraphics[width=12cm,height=9cm]{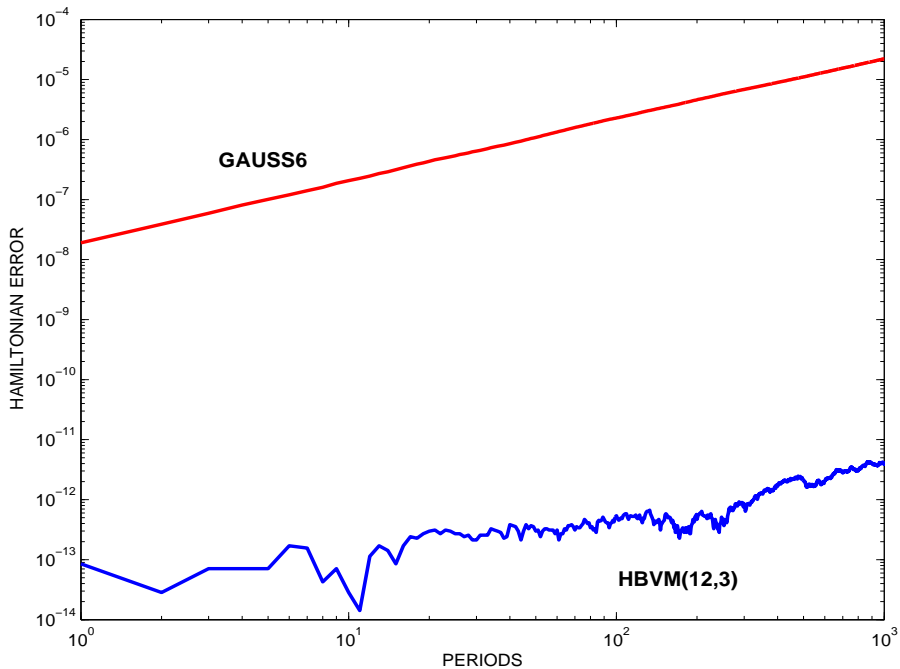}}
\caption{Hamiltonian error growth over 1000 periods with a
variable-step solution of problem (\ref{kepler})--(\ref{kepler0}),
$e=0.99$, with tolerance $tol=10^{-10}$.}
\label{fighbvmh}\end{figure}
\begin{figure}
\centerline{\includegraphics[width=12cm,height=9cm]{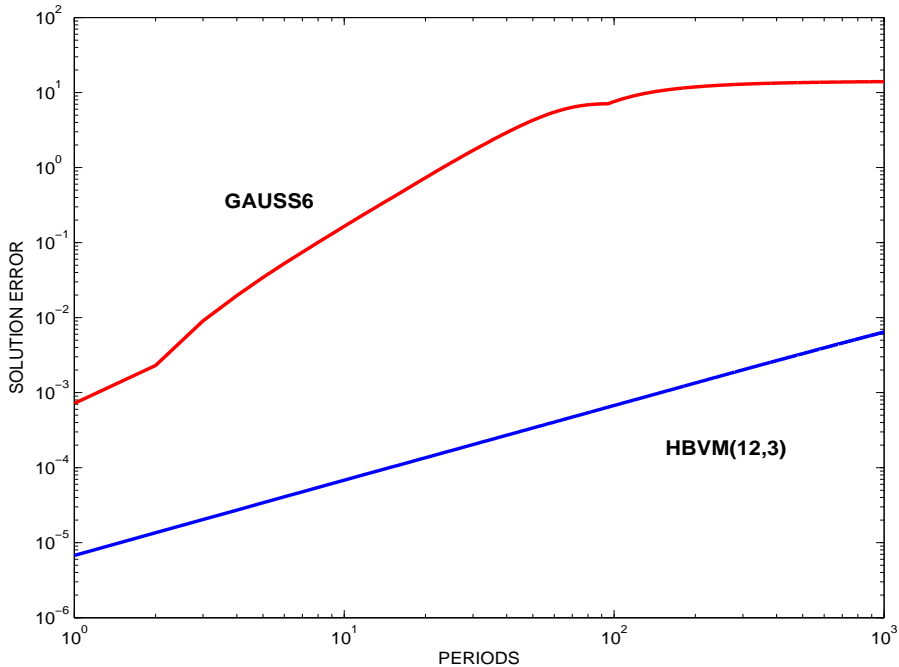}}
\caption{Error growth over 1000 periods with a variable-step
solution of problem (\ref{kepler})--(\ref{kepler0}), $e=0.99$,
with tolerance $tol=10^{-10}$.} \label{fighbvm}\end{figure}

We now consider the Kepler problem (see, e.g., \cite[p.\,9]{HLW}), with Hamiltonian
\begin{equation}\label{kepler}H([q_1,\, q_2,\, p_1,\, p_2]^T) = \frac{1}2(
p_1^2+p_2^2)-\frac{1}{\sqrt{q_1^2+q_2^2}},\end{equation} that,
when started at
\begin{equation}\label{kepler0}\pmatrix{cccc}1-e,&0, &0, &\sqrt{\frac{1+e}{1-e}}\endpmatrix^T,\end{equation}
has an elliptic periodic orbit of period $2\pi$ and eccentricity
$e\in[0,1)$. Moreover, in such a case, the (constant) value of the
Hamiltonian is $H_0=-\frac{1}2$. When $e$ is close to 0, the
problem is efficiently solved by using a constant stepsize.
However, it becomes more and more difficult as $e\rightarrow1$, so
that a variable-step integration would be more appropriate in this
case. In Figures~\ref{fighbvmh} and \ref{fighbvm}  we plot the
error growth in the Hamiltonian and in the solution, respectively,
over 1000 periods, in the case $e=0.99$, when using a {\em
standard} mesh selection strategy (e.g., like (1.1) in
\cite[p.\,303]{HLW}) with the (symplectic) Gauss-Legendre method
of order 6 and the HBVM (\ref{omegacprime}) with $r=3$ (then,
again of order 6), where the integral is approximated by means of
a Gauss formula with $k=12$ points, then having order 24, which is
sufficient to obtain, in this case, a practical energy
conservation (see \cite{BIT0}, and references therein, for full
details). The latter method is in general denoted by HBVM$(k,r)$
\cite{BIT0}, so that in the present case we consider the
HBVM(12,3) method.\footnote{According to what observed in
Remark~\ref{gaussrem}, the HBVM(3,3) method coincides with the
Gauss-Legendre method of order 6.} The tolerance used is
$tol=10^{-10}$. As one can see, the Gauss-Legendre formula
produces a {\em drift} in the Hamiltonian and a quadratic error
growth, whereas the HBVM(12,3) exhibits a negligible error in the
Hamiltonian and a linear error growth. This confirms that the
symplectic Gauss-Legendre method is more conveniently used with a
constant stepsize, whereas the energy preserving HBVM can be
profitably used with a {\em standard} mesh selection strategy.

\begin{table}
\caption{Statistics for the variable step implementation of
HBVM(12,3) on the Kepler problem (\ref{kepler})--(\ref{kepler0}),
$e=0.99$, with tolerance $tol=10^{-10}$.}\medskip \centerline{
\begin{tabular}{|r|r|r|}
\hline
  periods  &   error & points \\
\hline
   100  &6.75e-04     &   15300\\
   200  &1.36e-03     &   30600\\
   300  &2.04e-03     &   45900\\
   400  &2.72e-03     &   61200\\
   500  &3.41e-03     &   76500\\
   600  &4.10e-03     &   91800\\
   700  &4.79e-03     &  107100\\
   800  &5.48e-03     &  122400\\
   900  &6.17e-03     &  137700\\
  1000  &6.85e-03     &  153000\\
\hline
\end{tabular}}
\label{tabhbvm}
\end{table}

In Table~\ref{tabhbvm} we also list the number of points required
by the HBVM(12,3) method, with variable stepsize,
for covering an increasing number of periods: as one can easily
deduce from the listed data, 153 steps are required to cover each period.
In order to make clear the improvement over the symplectic sixth-order
Gauss method, it is enough to observe that,
in order to obtain a comparable accuracy, this method would require
approximately $2\cdot 10^5$ (constant) steps for each period!

\section{Conclusions}
We have reported three problems in Numerical Analysis considered
difficult for as long as half a century and which, in our opinion,
have been eventually resolved by dramatically changing  the
traditional approach suggested by experience. We reiterate that
this is certainly due to the fact that past studies were heavily
biased by the concept of continuity, whose peculiarities the
researchers have tried to import in the new context of Numerical
Analysis, where such problems are, instead, of discrete nature:
the two areas, i.e., the continuous one and the discrete one, are
often not overlapping.

Within this scenario, although it's certainly easier to follow the
tracks already drawn and consolidated in the literature, exploring
new routes sometimes allows to reach the goal more quickly and
innovatively, as picturesquely described in the tale of the {\em egg
of Columbus}.


\begin{thebibliography}{10}

\bibitem{Ale}
V.M. Alekseev, An estimate for the perturbations of the solution of
ordinary differential equations (Russian), {\em Vestn. Mosk. Univ.},
Ser.I, Math. Meh, 2 (1961) 28--36.

\bibitem{AIR}
D.S. Alexander, F. Iavernaro and A. Rosa, ``{\em Early Days in
Complex Dynamics},'' AMS-LMS History of Mathematics, (to appear).

\bibitem{AR} U. Ascher and S.  Reich, On some difficulties in integrating highly
oscillatory Hamiltonian systems, in ``Computational Molecular
Dynamics'', Lect. Notes Comput. Sci. Eng.  4, Springer, Berlin,
(1999),  281--296.

\bibitem{BM04} L.\,Brugnano and  C.\,Magherini, The {\tt BiM} code for the numerical
solution of ODEs, {\em J. Comput. Appl. Math.} 164--165 (2004)
145--158.

\bibitem{BM07} L.\,Brugnano and  C.\,Magherini, Blended implicit methods for solving
ODE and DAE problems, and their extension for second order problems,
{\em J. Comput. Appl. Math.} 205 (2007) 777--790.

\bibitem{BIT0}  L.\,Brugnano, F.\,Iavernaro and  D.\,Trigiante,
 The Hamiltonian BVMs (HBVMs) Homepage, {\tt arXiv:1002.2757}.

\bibitem{BIT1} L.\,Brugnano, F.\,Iavernaro and  D.\,Trigiante, Analisys
of Hamiltonian Boundary Value Methods (HBVMs): a class of
energy-preserving Runge-Kutta methods for the numerical solution of
polynomial Hamiltonian dynamical systems, {\em BIT} (2009)
submitted ({\tt arXiv:0909.5659}).

\bibitem{BIT2} L.\,Brugnano, F.\,Iavernaro and  D.\,Trigiante,
Hamiltonian Boundary Value Methods (Energy Preserving Discrete Line
Integral Methods), {\em Jour. of Numer. Anal. Industr. and Appl.
Math.} (2010) to appear  ({\tt arXiv:0910.3621}).

\bibitem{BIT3}  L.\,Brugnano, F.\,Iavernaro and  D.\,Trigiante,
Isospectral Property of HBVMs and their connections with Runge-Kutta
collocation methods, Preprint (2010)  ({\tt arXiv:1002.4394}).

\bibitem{BIT4} L.\,Brugnano, F.\,Iavernaro and  D.\,Trigiante, On the existence of
energy-preserving symplectic integrators based upon Gauss
collocation formulae, submitted (2010) ({\tt arXiv:1005.1930}).

\bibitem{BMT}
L.\, Brugnano, F. Mazzia and  D.\, Trigiante, Fifty Years of
Stiffness. Chapter~1 in {\em Recent Advances in Computational and Applied Mathematics},
T.E.\,Simos Ed., Springer, 2011.

\bibitem{BTstiff}
L. Brugnano and  D. Trigiante, On the characterization of Stiffness,
{\em Dynamics of Cont., Discr. Impuls. Systems} 2 (1996) 317--335.

\bibitem{BT07} L. Brugnano and D. Trigiante, A new mesh selection strategy for ODEs,
{\em Appl. Numer. Math.} 24 (1997) 1--21.

\bibitem{BTbook}
L. Brugnano and  D. Trigiante, ``{\em Solving ODEs by Linear
Multistep Initial and Boundary Value Methods},'' Gordon and Breach, Amsterdam,
1998.


\bibitem{CaCaMa07}
S.Capper, J.R. Cash and F. Mazzia, On the development of effective algorithms for
the numerical solution of singularly perturbed two-point boundary value problems,
{\em Int. Journ. of Comput. Science and Math.} 1 (2007) 42--57.

\bibitem{Cash1}
J. Cash, ``{\em Stable recursions. With applications to numerical
solution of stiff systems},'' Academic Press., London (1979).

\bibitem{CaMa05} J.R. Cash and F. Mazzia,  A new mesh selection algorithm,
based on conditioning, for two-point Boundary Value Codes,
{\em  Jour. Comput. Appl. Math.},  184 (2005) 362--381.

\bibitem{CaMa06} J.R. Cash and F. Mazzia, Hybrid mesh selection algorithms
based on conditioning for two-point boundary value problems, {\em
J. Numer. Anal. Ind. Appl. Math.}, 1 (2006) 81--90.

\bibitem{CaMa09} J.R. Cash and F. Mazzia, Conditioning and Hybrid Mesh
Selection Algorithms for Two-Point Boundary Value Problems, {\em Scalable
Computing: Practice and Experience}, 10 (4) (2009) 347--361.

\bibitem{CaMa} J.R. Cash and F. Mazzia,
 Algorithms for the solution of two-point boundary value
 problems,\quad
\url{http://www.ma.ic.ac.uk/~jcash/BVP_software/twpbvp.php}

\bibitem{Hir}
G.F. Curtiss and  J.O. Hirshfelder, Integration of stiff equations,
{\em Proc. Nat. Acad. Science U.S.} 38 (1952), 235--243.

\bibitem{DAbook}
G. Dahlquist and  \"{A}. Bj\"{o}rck, {\em ``Numerical Methods''},
Prentice Hall, Englewood Cliffs, N.J., 1974.

\bibitem{Da2}
G. Dahlquist, A special stability problem for linear multistep
methods, {\em BIT} 3 (1964), 27--43.

\bibitem{Da3}
G. Dahlquist, 33 years of instability, Part I,   {\em BIT} 25
(1985), 188--204.

\bibitem{Feng}
K. Feng, On difference schemes and symplectic geometry, in
Proceedings of the 5-th Intern. Symposium on differential geometry
\& differential equations, August 1984, Beijing (1985), 42--58.

\bibitem{WG}
W. Gautschi, Computational aspects of three-term recurrence
relations, {\em SIAM Rev.} 9 (1967), 24--82.

\bibitem{Gaubook}
W. Gautschi, ``\textit{Numerical Analysis, An introduction},''
Birkh\"{a}user Boston, Inc., Boston, MA,
1997.

\bibitem{Gold}
H. Goldstein, C.P. Poole and  J.L. Safko, ``\textit{Classical
Mechanics},'' Addison
Wesley, 2001.

\bibitem{G} O.\,Gonzalez, Time integration and discrete Hamiltonian
systems, {\em J. Nonlinear Sci.} 6 (1996), 449--467.

\bibitem{Higu} N. Guglielmi and E. Hairer (2007), Scholarpedia 2(11):2850,
revision \# 64754.\\
\url{http://www.scholarpedia.org/article/Stiff_delay_equations}

\bibitem{H} E.  Hairer, Symmetric projection methods for differential
equations on manifolds, {\em BIT} {40} (2000), 726--734.

\bibitem{Ha}  E. Hairer,  Energy-preserving variant of collocation
methods, {\em J. Numer. Anal. Ind. Appl. Math.},  to appear.

\bibitem{HLW} E.\,Hairer, C.\,Lubich and  G.\,Wanner, ``{\em Geometric
Numerical Integration. Structure-Preserving Algorithms for Ordinary
Differential Equations},'' Second ed., Springer, Berlin, 2006.

\bibitem{Hind}
A.C. Hindmarsh, ``{\em On Numerical Methods for Stiff Differential
Equations--Getting the Power to the People},'' Lawrence Livermore
Laboratory, UCRL-83259, 1979.

\bibitem{Hu}
W.H. Hundshorfer, The numerical solution of stiff initial value
problems: an analysis of one step methods, {\em CWI Tracts} 12
Amsterdam (1980).

\bibitem{im1} F. Iavernaro and  F. Mazzia, Solving  ordinary
differential equations by generalized adams methods: properties and
implementation techniques, {\em Appl. Numer. Math.}  (2--4) 28
(1998), 107--126.

\bibitem{im2} F. Iavernaro and  F. Mazzia, On the extension of the code
GAM for parallel computing, in EURO-PAR'99, Parallel Processing,
 Lecture Notes in Computer Science, 1685, Springer,
Berlin (1999), 1136--1143.

\bibitem{imt} F. Iavernaro, F. Mazzia and  D. Trigiante, Stability and
conditioning in Numerical Analysis, {\em JNAIAM} 1 (2006), 91--112.

\bibitem{IT1} F.\,Iavernaro and  D.\,Trigiante, Discrete conservative vector
fields induced by the trapezoidal method. {\em J. Numer. Anal. Ind.
Appl. Math.} 1 (2006), 113--130.

\bibitem{IT2} F.\,Iavernaro and  D.\,Trigiante, State-dependent symplecticity
and area preserving numerical methods, {\em J. Comput. Appl. Math.}
(2) 205 (2007), 814--825.

\bibitem{LakTri}
V. Lakshikantham and  D. Trigiante, ``{\em Theory of Difference
Equations. Numerical Methods and Applications},'' Second Edition,
Marcel Dekker, New York (2002).

\bibitem{LeRe} B.\,Leimkuhler and S.\,Reich, ``{\em Simulating Hamiltonian Dynamics}'',
Cambridge University Press 2004.

\bibitem{Ma03} F. Mazzia, Software for Boundary value Problems,
(2003), \\ \url{http://www.dm.uniba.it/~mazzia/bvp/index.html}

\bibitem{MaSeTr09} F. Mazzia, A. Sestini and D. Trigiante.,
The continous extension of the $B$-spline linear multistep metods
for BVPs on non-uniform meshes, {\em Appl. Numer. Math.}, 59(3--4)
(2009) 723--738.

\bibitem{MaTr04} F. Mazzia and D. Trigiante, A Hybrid Mesh Selection
Strategy Based on Conditioning for Boundary Value ODE Problems,
{\em Numerical Algorithms} 36 (2004) 169--187.


\bibitem{Ma1}
R.M.M. Mattheij, Characterizations of dominant and dominated
solutions of linear recursions, {\em Numer. Math.} 35 (1980), 421--442.

\bibitem{Ma2}
R.M.M. Mattheij, Stable computation of solutions of unstable linear
initial value recursions, {\em BIT} 22 (1982), 79--93.

\bibitem{testset} F. Mazzia and C. Magherini, Test Set for IVP Solvers, rel.\,2.4, (2008),\\
\url{http://dm.uniba.it/~testset/}

\bibitem{MQR} R.I.\,McLachlan, G.R.W.\,Quispel and  N.\,Robidoux,
Geometric integration using discrete gradient, {\em Phil. Trans. R.
Soc. Lond. A} 357 (1999), 1021--1045.

\bibitem{Olv1}
J. Oliver, The numerical solution of linear recurrence relations,
{\em Numer. Math.} 11 (1968), 349--360.

\bibitem{Olv2}
F.W.J. Olver, Numerical solution of second-order linear difference
equations, {\em J. Res. Nat. Bur. Standards} 71B (1967), 111-129.

\bibitem{multrep} L. Petzold et al., {\em Report on the First Multiscale
mathematics Workshop: First Steps towards a Roadmap}, 2004,\\
\url{http://www.mcs.anl.gov/~gropp/bib/reports/Multiscale-sept04.pdf}

\bibitem{Poincare-1886}
H. Poincar\'e, Sur les courbes d\'efinies par une \'equation
diff\'erentielle, {\em J. de Math\'{e}matiques pures et
appliqu\'{e}es}, quatri\`eme partie (4) 2 (1886),  151--217.


\bibitem{QMcL} G.R.W.\,Quispel and  D.I.\,McLaren, A new class of energy-preserving
numerical integration methods, {\em J. Phys. A: Math. Theor.} 41
(2008) 045206 (7pp).

\bibitem{Ruth} R.D. Ruth, A canonical integration technique, {\em IEEE Trans.
Nuclear Science} NS-30 (1983), 2669--2671.

\bibitem{SC} J.M.\,Sanz-Serna and  M.P.\,Calvo, ``{\em Numerical
Hamiltonian Problems},'' Chapman \& Hall, London, 1994.


\end{thebibliography}
\end{document}